\documentclass[10pt]{amsart}
\usepackage{graphicx}
\usepackage{amscd}
\usepackage{amsmath}
\usepackage{amsthm}
\usepackage{amsfonts}
\usepackage{amssymb}
\usepackage{mathrsfs}
\usepackage{enumerate}
\usepackage{amsrefs}
\usepackage{xcolor}
\usepackage[colorlinks, citecolor=blue, linkcolor=red, pdfstartview=FitB]{hyperref}
\usepackage[latin1]{inputenc}

\newcommand{\bB}{{\mathbb{B}}}
\newcommand{\bC}{{\mathbb{C}}}

\newcommand{\bM}{{\mathbb{M}}}
\newcommand{\bN}{{\mathbb{N}}}

\newcommand{\bS}{{\mathbb{S}}}

  \newcommand{\A}{{\mathcal{A}}}

\renewcommand{\H}{{\mathcal{H}}}

  \newcommand{\M}{{\mathcal{M}}}

\renewcommand{\S}{{\mathcal{S}}}

\newcommand{\fA}{{\mathfrak{A}}}

\newcommand{\fH}{{\mathfrak{H}}}

\newcommand{\fK}{{\mathfrak{K}}}

\newcommand{\fM}{{\mathfrak{M}}}
\newcommand{\fN}{{\mathfrak{N}}}

\newcommand{\fS}{{\mathfrak{S}}}
\newcommand{\fs}{{\mathfrak{s}}}
\newcommand{\fT}{{\mathfrak{T}}}

\newcommand{\rC}{\mathrm{C}}

\newcommand{\eps}{\varepsilon}
\renewcommand{\phi}{\varphi}

\newcommand{\upchi}{{\raise.35ex\hbox{$\chi$}}}

\newcommand{\ol}{\overline}

\newcommand{\qand}{\quad\text{and}\quad}
\newcommand{\qor}{\quad\text{or}\quad}

\newcommand{\dist}{\operatorname{dist}}

\newcommand{\id}{\operatorname{id}}

\newcommand{\re}{\operatorname{Re}}

\newcommand{\tr}{\operatorname{tr}}

\newtheorem{lemma}{Lemma}[section]
\newtheorem{theorem}[lemma]{Theorem}
\newtheorem{proposition}[lemma]{Proposition}

\theoremstyle{definition}

\newtheorem{example}[lemma]{Example}
\newtheorem{question}{Question}

\begin{document}
\author{Rapha\"el Clou\^atre}
\address{Department of Mathematics, University of Manitoba, Winnipeg, Manitoba, Canada R3T 2N2}
\email{raphael.clouatre@umanitoba.ca\vspace{-2ex}}
\thanks{The author was partially supported by an NSERC Discovery Grant}
\begin{abstract}
We investigate various notions of peaking behaviour for states on a $\rC^*$-algebra, where the peaking occurs within an operator system. We pay particularly close attention to the existence of sequences of elements forming an approximation of the characteristic function of a point in the state space. We exploit such characteristic sequences to localize the $\rC^*$-algebra at a given state, and use this localization procedure to verify  a variation of Arveson's hyperrigidity conjecture for arbitrary operator systems.
 \end{abstract}

\title[Peaking phenomena and local hyperrigidity]{Non-commutative peaking phenomena and a local version of the hyperrigidity conjecture}
\maketitle

\section{Introduction}

The study of unital closed subalgebras of continuous functions on some compact metric space has a rich history, involving questions and techniques ranging from complex function theory to measure theory and functional analysis. The resulting theory of \emph{uniform algebras} is quite extensive \cite{gamelin1969}, and we briefly recall one of its features that is relevant for our present purpose. 

Given a uniform algebra $\A\subset C(X)$ and a point $\xi\in X$, the linear functional on $\A$ of evaluation at $\xi$ can be represented as integration against some Borel probability  measure on $X$. Obviously, the point mass at $\xi$ is always a representing measure, and when it is the unique such, then the point $\xi$ is said to belong to the \emph{Choquet boundary} of $\A$.  Equivalently, we see that $\xi$ lies in the Choquet boundary of $\A$ if the corresponding evaluation functional admits a unique unital completely positive extension to $C(X)$. The Choquet boundary is a meaningful object to associate to a uniform algebra, for at least two reasons. 

First, it is known that the Choquet boundary is dense in the \emph{Shilov boundary} of $\A$, that is the smallest closed subset of $X$ on which every function in $\A$ attains its maximum modulus. In other words, the Choquet boundary provides a mechanism to identify the ``minimal" representation of the elements of $\A$ as functions on some compact metric space. This philosophy has been transplanted in the non-commutative context of unital operator algebras and operator systems, and has been exploited with great success to obtain a corresponding minimal representation of these objects inside a $\rC^*$-algebra. The visionary initial push to make this possible was furnished by Arveson \cite{arveson1969}, and it came to fruition several years later in the form of the \emph{$\rC^*$-envelope} \cite{hamana1979},\cite{MS1998},\cite{dritschel2005},\cite{arveson2008},\cite{davidson2015}. Central to the construction of this envelope are linear maps admitting a unique unital completely positive extension to the ambient $\rC^*$-algebra, very much in the spirit of the defining condition of the classical Choquet boundary. The theory surrounding the $\rC^*$-envelope has since flourished into a widely used invariant for the basic objects of non-commutative functional analysis \cite{hamana1999},\cite{blecher2001},\cite{FHL2016},\cite{CR2017}. 

The second meaningful aspect of the classical Choquet boundary that we wish to emphasize here, is that it can be characterized in a drastically different fashion from the one alluded to above, thus allowing for a wealth of additional interpretations. While the definition we gave above using completely positive extensions is somehow extrinsic, there is an equivalent definition that is intrinsic. Indeed, it is known \cite[Theorem II.11.3]{gamelin1969} that a point $\xi\in X$ lies in the Choquet boundary of $\A$ if and only if it is a \emph{peak point} for $\A$, in the sense that there is a function $\phi\in \A$ with the property that 
\[
|\phi(x)|<1=\phi(\xi)
\]
for every $x\in X$ such that $x\neq \xi$. Non-commutative extensions of this phenomenon have attracted significant interest in recent years \cite{hay2007},\cite{BHN2008},\cite{BR2011},\cite{blecher2013},\cite{BR2013},\cite{BR2014} and  beautiful generalizations of function theoretic results were obtained therein. Nevertheless, it has been noted (see the discussion following \cite[Proposition 5.7]{hay2007}, or \cite[Example 4.4]{BHN2008}) that the classical symbiosis between peak points and the Choquet boundary appears to collapse dramatically in this non-commutative interpretation. A robust connection between the non-commutative concepts is still lacking despite some subsequent efforts \cite{arveson2011},\cite{kleski2014bdry},\cite{NPSV2016}. Exploring a different foundation for such a connection is one of the purposes of the paper.

The other purpose of this work is to address the so-called \emph{hyperrigidity conjecture} of Arveson \cite{arveson2011}.
To motivate this conjecture, we recall a classical result of Korovkin \cite{korovkin1953}. For each $n\in \bN$, let $\Phi_n:C[0,1]\to C[0,1]$ be a positive linear map and assume that
\[
\lim_{n\to \infty}\|\Phi_n(f)-f\|=0
\]
for every $f\in \{1,x,x^2\}$. Then, it must be the case that
\[
\lim_{n\to \infty}\|\Phi_n(f)-f\|=0
\]
for every $f\in C[0,1]$. This striking rigidity phenomenon was elucidated by \v Sa\v skin in \cite{saskin1967}, who showed that the crucial property at play here is that the interval $[0,1]$ consists entirely of peak points for the subspace generated by $\{1,x,x^2\}$. One then wonders whether this is a manifestation of a general phenomenon which is valid in the non-commutative context as well. More precisely, let $\fS$ be an operator system. Arveson conjectured that if the non-commutative Choquet boundary of $\fS$ is ``maximal" then $\fS$ is \emph{hyperrigid}, in the sense that for any injective $*$-representation 
\[
\pi:\rC^*(\fS)\to B(\H_\pi)
\]
and for any sequence of unital completely positive linear maps
\[
\Phi_n:B(\H_\pi)\to B(\H_\pi), \quad n\in \bN
\]
such that
\[
\lim_{n\to \infty}\|\Phi_n(\pi(s))-\pi(s)\|=0, \quad s\in \fS
\]
we must have
\[
\lim_{n\to \infty}\|\Phi_n(\pi(a))-\pi(a)\|=0, \quad a\in \rC^*(\fS).
\]
The notion of hyperrigidity is deeply rooted in various parts of operator algebras and operator theory; for instance, in the appropriate context it has been shown to be equivalent to the \emph{Arveson-Douglas essential normality conjecture} involving quotient modules of the Drury-Arveson space \cite[Theorem 4.12]{kennedyshalit2015}. This conjecture has generated a flurry of activity and has witnessed spectacular progress recently (see \cite{englis2015} and \cite{DTY2016} and the references therein). Although we will not discuss the essential normality conjecture in any more detail here, the aforementioned equivalence makes the hyperrigidity conjecture all the more tantalizing.

At present the hyperrigidity conjecture is still unresolved in full generality. It has been verified in the case where $\rC^*(\fS)$ has countable spectrum \cite[Theorem 5.1]{arveson2011} and in the case where $\rC^*(\fS)$ is commutative \cite[Corollary 6.8]{DK2016}. Some partial results in other contexts have also appeared in \cite{kleski2014hyper},\cite{clouatre2017unp}. The techniques used therein are geared towards establishing the unique extension property of unital completely positive maps. In this paper, we choose a different route. We wish to exhibit a link between hyperrigidity and a seemingly neglected facet of the non-commutative Choquet boundary, namely peaking phenomena; such a link would constitute a first step in a faithful adaptation of the aforementioned approximation results of Korovkin and \v Sa\v skin. More precisely, we will witness hyperrigidity in some ``local" sense, where the precise localization procedure is accomplished via peaking states. It should be noted here that Arveson himself had early results concerning local hyperrigidity  \cite[Theorem 11.2]{arveson2011} in the commutative setting. Our contribution is to show that Arveson's commutative local hyperrigidity, once appropriately interpreted with the help of peaking states, holds in full generality.

We now describe the organization of the paper more precisely. In Section \ref{S:prelim}, we gather the necessary background material and prove some preliminary results that are required throughout.

In Section \ref{S:peaking}, we introduce a notion of peaking states, which goes as follows. Let $\fA$ be a unital $\rC^*$-algebra and let $\fS \subset \fA$ be an operator system. We say that a state $\psi$ on $\fA$ is $\fS$-\emph{peaking} if there is a self-adjoint element $s\in \fS$ such that $\|s\|=1$ and with the property that 
\[
\psi(s)=1>|\phi(s)|
\]
for every state $\phi$ on $\fA$ such that $\phi\neq \psi$. In Theorem \ref{T:peakingexposed} we obtain the following characterization of these objects; it is noteworthy in view of the classical result of Klee \cite{klee1958} which guarantees the existence of a large supply of weak-$*$ exposed points. 

\begin{theorem}\label{T:main1}
Let $\fA$ be a unital $\rC^*$-algebra and let $\fS\subset \fA$ be an operator system.  Let $\psi$ be a state on $\fA$. Then, the following statements are equivalent.
\begin{enumerate}

\item[\rm{(i)}] The state $\psi$ is $\fS$-peaking.

\item[\rm{(ii)}] The state $\psi$ has the unique extension property with respect to $\fS$, and
the restriction $\psi|_{\fS}$ is a weak-$*$ exposed point of  the set of self-adjoint functionals on $\fS$ with norm at most $1$.
\end{enumerate}
\end{theorem}
We consider another flavour of peaking behaviour, heavily inspired by the work of Hay \cite{hay2007}. A projection $p\in \fA^{**}$ is said to be \emph{$\fS$-peaking} if there is a self-adjoint element $s\in \fS$ with $\|s\|=1$ such that $ sp=p$ and $| \phi(s)|<1$ whenever $\phi$ is a state on $\fA$ with support projection $\fs_\phi\in \fA^{**}$ orthogonal to $p$.  We elucidate the relationship between $\fS$-peaking states and $\fS$-peaking projections in the following (Theorem \ref{T:peakHay}).

\begin{theorem}\label{T:main2}
Let $\fA$ be a unital $\rC^*$-algebra and let $\fS\subset \fA$ be an operator system. Then, a pure state $\omega$ on $\fA$ is $\fS$-peaking if and only if its support projection $\fs_\omega\in \fA^{**}$ is $\fS$-peaking.
\end{theorem}

In Section \ref{S:aspeaking}, in preparation for studying a local version of hyperrigidity and motivated by concrete examples, we introduce a more global type of peaking phenomenon for states. Roughly speaking, we wish to find, within a unital $\rC^*$-algebra $\fA$, an approximate version of the characteristic function corresponding to a given state $\psi$ on $\fA$. More precisely, given a sequence $(\Delta_n)_n$ in $\fA$ such that $\|\Delta_n\|=1$ for every $n\in \bN$, we say that $(\Delta_n)_n$ is a \emph{characteristic sequence} for $\psi$ if 
\[
\lim_{n\to \infty }\psi(\Delta_n)=1
\]
and
\[
\limsup_{n\to\infty}\|\Delta_n^* a \Delta_n\|\leq |\psi(a)|, \quad a\in \fA.
\]
We show that states that admit a characteristic sequence share several properties with $\fS$-peaking states and can be thought of as being ``asymptotically $\fS$-peaking" (Theorem \ref{T:gapeakSchwarz}).

\begin{theorem}\label{T:main3}
Let $\fA$ be a unital $\rC^*$-algebra and let $\psi$ a state on $\fA$. Let $(\Delta_n)_n$ be a characteristic sequence for $\psi$. 
The following statements hold.
\begin{enumerate}

\item[\rm{(1)}]  We have that 
\[
\lim_{n\to\infty}\|\Delta_n a \Delta_n^*\|=|\psi(a)|
\]
for every $a\in \fA$.

\item[\rm{(2)}]  We have that 
\[
\limsup_{n\to\infty}|\phi(\Delta_n)|<1
\]
for every state $\phi$ on $\fA$ such that $\phi\neq \psi$.

\item[\rm{(3)}] The state $\psi$ is pure. 

\item[\rm{(4)}] Let $\fS\subset \fA$ be an operator system and assume that $\Delta_n\in \fS$ for each $n\in \bN$. Then, $\psi$ has the unique extension property with respect to $\fS$.
\end{enumerate}
\end{theorem}

The asymptotic notion of characteristic sequence is flexible enough to occur in natural and important examples such as the higher-dimensional Toeplitz algebra, which is a ubiquitous object in multivariate operator theory, dilation theory and function theory (Example \ref{E:Toeplitz}).

Finally, in Section \ref{S:local} we exploit characteristic sequences to localize a $\rC^*$-algebra at a given state, and we establish the following (Theorem \ref{T:gapeaklocalhr}).

\begin{theorem}\label{T:main4}
Let $\fA$ be a unital  $\rC^*$-algebra and let $\fS\subset \fA$ be an operator system. Let $\pi:\fA\to B(\H)$ be a unital $*$-representation and let $\Pi:\fA\to B(\H)$ be a unital completely positive extension of $\pi|_{\fS}$. Let $\psi$ be a state on $\fA$ which admits a characteristic sequence $(\Delta_n)_n$ in $\fS$. Then, we have
\[
\lim_{n\to\infty} \|\pi(\Delta_n)^*(\Pi(a)-\pi(a))\pi(\Delta_n)\|=0
\]
for every $a\in \fA$.
\end{theorem}

In view of part (3) of Theorem \ref{T:main3}, this result can be viewed as supporting evidence for Arveson's conjecture, as it establishes a local form of the desired statement. Note that although the conclusion is merely local, so is the hypothesis. Thus, the theorem appears to be relevant even in cases where the conjecture has been verified. Moreover, it generalizes Arveson's local hyperrigidity theorem \cite[Theorem 11.1]{arveson2011} that applies to commutative $\rC^*$-algebras, as we show in detail at the end of the paper.

\textbf{Acknowledgements.} The author wishes to thank the referee for several insightful suggestions.

\section{Background and preliminary results}\label{S:prelim}

\subsection{ Operator systems and completely positive maps}\label{SS:ucp}
Throughout the paper, $\H$ will denote a Hilbert space and $B(\H)$ will denote the bounded linear operators on it. We now briefly review the basics of operator systems and completely positive maps. The reader may wish to consult \cite{paulsen2002} for greater detail.

Let $\fA$ be a unital $\rC^*$-algebra. A unital self-adjoint subspace $\fS\subset \fA$ is called an \emph{operator system}. As is well-known, it is possible to define operator systems in a more abstract fashion that does not rely on an embedding inside of a unital $\rC^*$-algebra \cite[Theorem 4.4]{CE1977}, but the previous concrete definition is sufficient for our purposes. For each $n\in \bN$, we may view the space $\bM_n(\fS)$ of $n\times n$ matrices with entries from $\fS$ as a self-adjoint subspace of $\bM_n(\fA)$, and in particular there is an associated notion of positivity for elements of $\bM_n(\fS)$. 

A linear map $\phi:\fS\to B(\H)$ is said to be \emph{completely positive} if $\phi^{(n)}$ is positive for every $n\in \bN$. Here, given $n\in \bN$, we denote by $\phi^{(n)}:\bM_n(\fS)\to B(\H^{(n)})$ the induced linear map defined as
\[
\phi^{(n)}([s_{ij}]_{i,j})=[\phi(s_{ij})_{i,j}], \quad [s_{ij}]_{i,j}\in \bM_n(\fS).
\]
If $\phi$ is a unital completely positive map on $\fA$, then it satisfies the \emph{Schwarz inequality}:
\[
\phi^{(n)}(A)^*\phi^{(n)}(A)\leq \phi^{(n)}(A^*A)
\]
for each $A\in \bM_n(\fA)$ and $n\in \bN$.
We will use this inequality repeatedly throughout the paper. 

\subsection{Extensions of completely positive maps and hyperrigidity}\label{SS:hyper}

Let $\fA$ be a unital $\rC^*$-algebra and let $\fS\subset \fA$ be an operator system. Given a unital completely positive map $\phi:\fS\to B(\H)$, by virtue of Arveson's extension theorem there is a unital completely positive map $\Phi:\fA\to B(\H)$ which extends $\phi$. In general, the extension is not unique. Accordingly, we say that a unital completely positive map $\psi:\fA\to B(\H)$ has the \emph{unique extension property} with respect to $\fS$ if it is the unique completely positive extension of $\psi|_{\fS}$ to $\fA$. We advise the reader to exercise some care: in other works (such as \cite{arveson2008}) the use of the terminology ``unique extension property" is reserved for $*$-representations of $\rC^*(\fS)$. In our context, we found our more lenient definition to be more convenient, but the reader should keep this discrepancy in mind throughout. The unique extension property is closely related to the phenomenon of hyperrigidity for operator systems that was defined in the introduction. Indeed, it follows from \cite[Theorem 2.1]{arveson2011} that an operator system  $\fS$ is hyperrigid if and only if every unital $*$-representation of $\rC^*(\fS)$ has the unique extension property with respect to $\fS$.

In the special case where we have an irreducible $*$-representation $\pi:\fA\to B(\H)$ which has the unique extension property with respect to $\fS$, then we say that $\pi$ is a \emph{boundary representation for $\fS$}. This is motivated by the definition of the Choquet boundary of a uniform algebra, as described in the introduction. A basic problem is to determine to which extent the boundary representations of $\rC^*(\fS)$ determine the hyperrigidity of $\fS$, in the spirit of the theorems of Korovkin and \v Sa\v skin mentioned earlier. This is the content of the following conjecture formulated in \cite{arveson2011}.

\vspace{3mm}

\emph{Arveson's hyperrigidity conjecture.} An operator system $\fS$ is hyperrigid if and only if every irreducible $*$-representation of $\rC^*(\fS)$ is a boundary representation for $\fS$.

\vspace{3mm}

In the case where $\fS$ is separable and $\rC^*(\fS)$ is commutative, Arveson managed to establish a local version of this conjecture \cite[Theorem 11.1]{arveson2011}. Since this is one of the main motivations for our work, we wish to state this particular result precisely. First, we set up some notation.

Let $(X,\rho)$ be a compact metric space and let $\fS\subset C(X)$ be an operator system (a \emph{function system}) such that $\rC^*(\fS)=C(X)$. We know that the irreducible $*$-representations of $C(X)$ are precisely the characters of evaluation at points of $X$. In particular, the character of evaluation at $x\in X$ is a boundary representation for $\fS$ if and only if $x$ lies in the Choquet boundary of $\fS$  (the Choquet boundary of a function system is defined in a manner completely analogous to that of a uniform algebra).

Let $\H$ be a separable Hilbert space and let $\pi:\rC^*(\fS)\to B(\H)$ be a unital $*$-representation. Put $\fM_\pi=\pi(\rC^*(\fS))''$. Then, there is a weak-$*$ homeomorphic $*$-isomorphism $\Theta:\fM_\pi\to L^\infty(X,\mu)$ for some Borel probability measure $\mu$ with support equal to $X$. For each $\Omega\subset X$, we let $\chi_\Omega\in L^\infty(X,\mu)$ denote the characteristic function of $\Omega$. Moreover, for $x\in X$ and $\delta>0$, we let
\[
B(x,\delta)=\{y\in X:\rho(x,y)<\delta\}.
\]
Then, for $x\in X$ and $\delta>0$ we put
\[
E_\pi(x,\delta)=\Theta^{-1}(\chi_{B(x,\delta)})\in \fM_\pi.
\]
The following is a reformulation of \cite[Theorem 11.1]{arveson2011}, which can be easily extracted from its proof.

\begin{theorem}\label{T:arvlochr}
Let $X$ be a compact metric space and let $\fS\subset C(X)$ be an operator system such that $\rC^*(\fS)=C(X)$. Let $\H$ be a separable Hilbert space, let $\pi:\rC^*(\fS)\to B(\H)$ be a unital $*$-representation and let $\Pi:\rC^*(\fS)\to B(\H)$ be a unital completely positive extension of $\pi|_{\fS}$. Let $x\in X$ be a point in the Choquet boundary of $\fS$. Then, 
\[
\lim_{\delta\to 0}\|(\pi(a)-\Pi(a))E_\pi(x,\delta) \|=0, \quad a\in C(X).
\]
\end{theorem}

We should point out that even though the hyperrigidity conjecture has been verified recently in the commutative case \cite{DK2016}, it is not entirely clear to us how to obtain Theorem \ref{T:arvlochr} as a consequence of that seemingly stronger result. 

\subsection{States and support projections}
Let $\fA$ be a unital $\rC^*$-algebra and let $\fS\subset \fA$ be an operator system.
We will frequently be dealing with scalar valued unital completely positive maps on $\fS$, which are called \emph{states}. We will denote by $\S(\fS)$ the \emph{state space} of $\fS$. A state on $\fS$ is \emph{pure} if it is an extreme point of $\S(\fS)$. 

Assume now that we have a state $\phi$ defined on the unital $\rC^*$-algebra $\fA$. The associated GNS \emph{representation} is the essentially unique triple $(\sigma_\phi, \fH_\phi, \xi_\phi)$ consisting of a Hilbert space $\fH_\phi$,  a unital $*$-representation $\sigma_\phi:\fA\to B(\fH_\phi)$ and a cyclic unit vector $\xi_\phi\in \fH_\phi$ with the property that
\[
\phi(a)=\langle \sigma_\phi(a) \xi_\phi, \xi_\phi\rangle, \quad a\in \fA.
\]
It is well-known \cite[Proposition 2.5.4]{dixmier1977} that $\sigma_\phi$ is irreducible precisely when $\phi$ is a pure state on $\fA$. 

One important consequence of the Schwarz inequality applied to the state $\phi$ is that if an element $a\in \fA$ is such that $\|a\|=1$ and $|\phi(a)|=1$, then $|\phi(a)|^2=\phi(a^*a)$ and consequently $a$ belongs to the \emph{multiplicative domain of $\phi$}, in the sense that
\[
\phi(ba)=\phi(b)\phi(a), \quad b\in \fA.
\]
This classical observation is due to Choi \cite[Theorem 3.1]{choi1974} and will be used repeatedly throughout.
Another consequence of the Schwarz inequality is that the set
\[
\{a\in \fA:\phi(a^*a)=0\}
\]
coincides with
\[
\{a\in \fA: \phi(ba)=0, b\in \fA\}
\]
and in particular it is a norm closed left ideal of $\fA$. 

Recall that the bidual $\fA^{**}$ can be identified with the universal enveloping von Neumann algebra of $\fA$, and it contains $\fA$ as a $\rC^*$-subalgebra \cite[Theorem III.2.4]{takesaki2002}. There is a unique weak-$*$ continuous state $\widehat \phi$ on $\fA^{**}$ which extends $\phi$. Then, the set
\[
\fN_\phi=\{x \in \fA^{**}:\widehat \phi(x^* x)=0\}=\{x\in \fA^{**}:\widehat\phi(y x)=0, y\in \fA^{**}\}
\]
is a weak-$*$ closed left ideal of the von Neumann algebra $\fA^{**}$. We infer \cite[Proposition 1.10.1]{sakai1971} that there is a projection $\fs_\phi\in \fA^{**}$ such that
\[
\fN_\phi=\fA^{**}(I-\fs_\phi).
\]
This projection $\fs_\phi$ is called the \emph{support projection} of $\phi$ \cite[Definition 1.14.2]{sakai1971}. We note that $\widehat \phi(\fs_\phi)=1$ and hence $\fs_\phi$ lies in the multiplicative domain of $\widehat \phi$. Thus,
\[
\widehat\phi(\fs_\phi x \fs_\phi)=\widehat\phi(x), \quad x\in \fA^{**}.
\]
We note also that the restriction of $\widehat\phi$ to $\fs_\phi\fA^{**} \fs_\phi$ is \emph{faithful}, that is $\widehat\phi(x^* x)>0$ for every non-zero element $x\in \fA^{**}\fs_\phi$.

 The next result elucidates some properties of support projections of pure states.
 
\begin{lemma}\label{L:distinctpure}
Let $\fA$ be a unital $\rC^*$-algebra and let $\omega$ be a pure state on $\fA$. Then, $\fs_\omega \fA^{**} \fs_\omega=\bC \fs_\omega$. In particular, if $\phi$ is a state on $\fA$ with $\widehat\phi(\fs_\omega)=1$, then $\phi=\omega$.
\end{lemma}
\begin{proof}
It is shown in the proof of \cite[Proposition III.2.16]{takesaki2002} that $\fs_\omega \fA^{**} \fs_\omega=\bC \fs_\omega$. In particular, we see that
\[
\fs_\omega x \fs_\omega=\widehat\omega(x) \fs_\omega, \quad x\in \fA^{**}.
\]
If $\widehat\phi(\fs_\omega)=1$, then $\fs_\omega$ belongs to the multiplicative domain of $\widehat \phi$ and we find
\[
\widehat \omega(x)=\widehat \omega(x)\widehat \phi(\fs_\omega)=\widehat \phi(\fs_\omega x \fs_\omega)=\widehat \phi(x)
\]
for every $x\in \fA^{**}$. We conclude that $\phi=\omega$.

\end{proof}

The next fact we will need concerning support projections is likely well-known to experts, but we could not locate a convenient reference. Accordingly, we provide a detailed proof, as support projections will play a key role in our analysis of peaking states in the next section. Before we state the result, recall that two states $\phi,\psi$ on $\fA$ are said to be \emph{orthogonal} if $\|\psi-\phi\|=2$. This is equivalent to their support projections $\fs_\phi,\fs_\psi\in \fA^{**}$ being orthogonal \cite[Theorem III.4.2]{takesaki2002}.
\begin{proposition}\label{P:purestateortho}
Let $\fA$ be a unital $\rC^*$-algebra. The following statements are equivalent.
\begin{enumerate}
\item[\rm{(i)}] $\fA$ is commutative.

\item[\rm{(ii)}] Any two distinct pure states are orthogonal.
\end{enumerate}
\end{proposition}
\begin{proof}
If $\fA$ is commutative, then the pure states on $\fA$ are characters. Since the character space of $\fA$ is compact and Hausdorff, Urysohn's lemma implies that distinct characters are orthogonal. 

Conversely, recall that the direct sum of all irreducible $*$-representations of $\fA$ is injective \cite[Corollary I.9.11]{davidson1996}. Thus, if $\fA$ is not commutative, we can find an irreducible $*$-representation $\pi:\fA\to B(\H)$ where $\H$ has dimension at least two. Let $\xi,\eta\in \H$ be two unit vectors which are not orthogonal. Define two states $\omega_\xi$ and $\omega_\eta$ on $\fA$ as
\[
\omega_\xi(a)=\langle \pi(a)\xi,\xi\rangle \qand \omega_\eta(a)=\langle \pi(a)\eta,\eta\rangle
\]
for every $a\in \fA$. Since $\pi$ is irreducible, both $\xi$ and $\eta$ are cyclic vectors and it is easily verified that $(\pi,\H,\xi)$ and $(\pi,\H,\eta)$ are unitarily equivalent to the GNS representations of $\omega_\xi$ and $\omega_\eta$ respectively. We conclude that $\omega_\xi$ and $\omega_\eta$ are pure. As shown in the proof of \cite[Theorem III.2.4]{takesaki2002}, we may represent $\fA^{**}$ faithfully on some Hilbert space $\fH$ with the property that there are non-orthogonal unit vectors $\xi',\eta'\in \fH$ such that
\[
\omega_\xi(a)=\langle a\xi',\xi' \rangle \qand \omega_\eta(a)=\langle a\eta',\eta' \rangle
\]
for every $a\in \fA$. In particular, we have that
\[
\widehat{\omega_\xi}(x)=\langle x\xi',\xi' \rangle \qand \widehat{\omega_\eta}(x)=\langle x\eta',\eta' \rangle
\]
for every $x\in \fA^{**}$. Using that $\widehat{\omega_\xi}(\fs_{\omega_\xi})=1$, we find that $\langle \fs_{\omega_\xi}\xi',\xi'  \rangle=1$ and thus $\xi'\in \fs_{\omega_\xi}\fH$. Likewise, we see that $\eta'\in \fs_{\omega_\eta}\fH$. Since $\xi'$ and $\eta'$ are not orthogonal, we infer that the support projections $\fs_{\omega_\xi}$ and  $\fs_{\omega_\eta}$ are not orthogonal, and thus neither are the states $\omega_\xi$ and $\omega_\eta$.
\end{proof}

\section{Peaking states}\label{S:peaking}

In this section, we consider peaking phenomena in non-commutative contexts. There has been some previous work done on this topic; we mention for instance the notions of peaking $*$-representations \cite{arveson2011}, \cite{NPSV2016} or of peaking projections \cite{hay2007}. We opt here to focus our attention on states. After having characterized peaking states, we will clarify their relationship to peaking representations and peaking projections. 

Let $\fA$ be a unital $\rC^*$-algebra and let $\fS \subset \fA$ be an operator system. We say that a state $\psi$ on $\fA$ is $\fS$-\emph{peaking} if there is a self-adjoint element $s\in \fS$ such that $\|s\|=1$ and with the property that 
\[
\psi(s)=1>|\phi(s)|
\]
for every state $\phi$ on $\fA$ such that $\phi\neq \psi$. We then say that $\psi$ \emph{peaks on} $s$.
This simple property imposes strong conditions on the state $\psi$, as we now show.

\begin{lemma}\label{L:peakuep}
Let $\fA$ be a unital $\rC^*$-algebra and let $\fS\subset \fA$ be an operator system. Then, every state on $\fA$ which is $\fS$-peaking must be pure and must have the unique extension property with respect to $\fS$.
\end{lemma}
\begin{proof}
Let $\psi$ be a state on $\fA$ which is $\fS$-peaking, so that we can find a self-adjoint element $s\in \fS$ with $\|s\|=1$ and such that 
\[
\psi(s)=1>| \phi(s)|
\]
for every state $\phi$ on $\fA$ such that $\phi\neq \psi$. Fix such a state $\phi$.  In particular, we observe that $\psi|_{\fS}\neq \phi|_{\fS}$. This shows that $\psi$ has the unique extension property with respect to $\fS$. Moreover, if we write 
\[
\psi=\frac{1}{2}(\phi_1+\phi_2)
\]
for some states $\phi_1,\phi_2$ on $\fA$, then we must have $\phi_1(s)=\phi_2(s)=1$. By choice of the element $s$, this forces $\psi=\phi_1=\phi_2$ and thus $\psi$ is pure.
\end{proof}

We see that a (necessarily pure) state on $\fA$ which is $\fS$-peaking must admit a so-called \emph{determining element} in $\fS$. This notion apparently has a quantum mechanical interpretation as discussed in \cite{hamhalter2001}, and it is reminiscent of the following familiar convexity concept.  

Let $X$ be a normed space and let $Q$ be a weak-$*$ closed convex subset of the closed unit ball of the dual space $X^*$. Assume that $0\in Q$. Recall that a non-zero element $\Lambda_0\in Q$ is a  \emph{weak-$*$ exposed point}  if there is $x_0\in X$ with the property that
\[
\re \Lambda_0(x_0)=1>\re \Lambda(x_0)
\]
whenever $\Lambda\in Q$ and $\Lambda\neq \Lambda_0$. Arguing as in the proof of Lemma \ref{L:peakuep}, it is straightforward to verify that weak-$*$ exposed points are necessarily extreme points.  Using this concept, we can  characterize $\fS$-peaking states.

\begin{theorem}\label{T:peakingexposed}
Let $\fA$ be a unital $\rC^*$-algebra and let $\fS\subset \fA$ be an operator system.  Let $\psi$ be a state on $\fA$. Then, the following statements are equivalent.
\begin{enumerate}

\item[\rm{(i)}] The state $\psi$ is $\fS$-peaking.

\item[\rm{(ii)}] The state $\psi$ has the unique extension property with respect to $\fS$, and
the restriction $\psi|_{\fS}$ is a weak-$*$ exposed point of  the set of self-adjoint functionals on $\fS$ with norm at most $1$.
\end{enumerate}
\end{theorem}
\begin{proof}
Throughout the proof, we let $Q$ denote the set of self-adjoint functionals on $\fS$ with norm at most $1$.

Assume that $\psi$ has the unique extension property with respect to $\fS$ and that $\psi|_{\fS}$ is a weak-$*$ exposed point of $Q$. Then,  there is $x_0\in \fS$ with the property that
\[
\re \psi(x_0)=1>\re \Lambda(x_0)
\]
for every $\Lambda\in Q$ such that $\Lambda\neq \psi|_{\fS}$. Upon replacing $x_0$ by $(x_0+x_0^*)/2$ if necessary, we may assume that $x_0$ is self-adjoint. Hence, we see that
\[
\psi(x_0)=1>\Lambda(x_0)
\]
for every $\Lambda\in Q$ such that $\Lambda\neq \psi|_{\fS}$. Since $-\Lambda\in Q$ whenever $\Lambda\in Q$, we see that
\begin{equation}\label{EQ:exppeak}
\psi(x_0)=1>|\Lambda(x_0)|
\end{equation}
for every $\Lambda\in Q$ such that $\Lambda\neq \pm\psi|_{\fS}$. Thus $|\phi(x_0)|\leq 1$ if $\phi$ is a state on $\fS$, which implies that $\|x_0\|=1$ as $x_0$ is assumed to be self-adjoint. Finally, the unique extension property of $\psi$ with respect to $\fS$ along with Equation (\ref{EQ:exppeak}) shows that $\psi$ is $\fS$-peaking.

Conversely, assume that $\psi$ is $\fS$-peaking, so that there is a self-adjoint element $x_0\in \fS$ with $\|x_0\|= 1$ such that
\[
\psi(x_0)=1>|\phi(x_0)|
\]
whenever $\phi$ is a state on $\fA$ with $\phi\neq \psi$. By Lemma \ref{L:peakuep}, we see that $\psi$ has the unique extension property with respect to $\fS$. It remains only to show that the restriction $\psi|_{\fS}$ is a weak-$*$ exposed point of $Q$. For that purpose, we will show that
\[
|\Lambda(x_0)|<1
\]
whenever $\Lambda\in Q$ with $\Lambda\neq \pm\psi|_{\fS}$. Clearly, it suffices to deal with the case where $\|\Lambda\|=1$ and $\Lambda$ is not of the form $\pm \phi$ for some state $\phi$ on $\fS$. Let $\Theta$ be a functional on $\fA$ with $\|\Theta\|=1$ which agrees with $\Lambda$ on $\fS$. Upon replacing $\Theta$ by its real part if necessary, we may  assume that $\Theta$ is self-adjoint. By means of the Jordan decomposition \cite[Theorem 1.7.12]{conway2000}, we may write 
\[
\Theta=\Theta_1-\Theta_2
\]
for some positive linear functionals $\Theta_1$ and $\Theta_2$ on $\fA$ with the property that 
\[
1=\|\Theta\|=\|\Theta_1\|+\|\Theta_2\|.
\]
Since $\Lambda$ is not of the form $\pm \phi$ for some state $\phi$ on $\fS$, we see that $\pm\Theta_k|_{\fS}\neq \Lambda$ and thus $\|\Theta_k\|\neq 0$ for each $k=1,2$.  Note now that $\Theta_k/\|\Theta_k\|$ is a state on $\fA$ for each $k=1,2$, and
\[
\Theta=\|\Theta_1\| \frac{\Theta_1}{\|\Theta_1\|}+\|\Theta_2\| \frac{(-\Theta_2)}{\|\Theta_2\|}
\]
whence
\[
\Lambda=\|\Theta_1\| \frac{\Theta_1|_{\fS}}{\|\Theta_1\|}+\|\Theta_2\| \frac{(-\Theta_2|_{\fS})}{\|\Theta_2\|}.
\]
Using that $\|\Lambda\|=1$, a straightforward verification yields that either 
\[
\frac{\Theta_1|_{\fS}}{\|\Theta_1\|}\neq \psi|_{\fS} \qor \frac{\Theta_2|_{\fS}}{\|\Theta_2\|}\neq  \psi|_{\fS}.
\] 
By choice of $x_0$, we conclude that
\[
|\Lambda(x_0)|\leq \|\Theta_1\| \frac{|\Theta_1(x_0)|}{\|\Theta_1\|}+\|\Theta_2\| \frac{|\Theta_2(x_0)|}{\|\Theta_2\|}<\|\Theta_1\|+\|\Theta_2\|=1.
\]

\end{proof}

It is natural to wonder just how common states having the unique extension property happen to be. Irreducible $*$-representations with the unique extension property with respect to an operator system are  abundant enough to generate the $\rC^*$-envelope \cite{dritschel2005},\cite{arveson2008},\cite{davidsonkennedy2015}, but a corresponding statement for states is unknown to us. In fact, \cite[Theorem 3.12 and Corollary 5.9]{clouatre2017unp} show that similar questions are closely related to the hyperrigidity conjecture itself.

If a pure state $\psi$ on $\fA$ has the unique extension property with respect to $\fS$, then the restriction $\psi|_{\fS}$ is pure as well \cite[Lemma 3.9]{clouatre2017unp}. By the previous theorem, determining whether $\psi$ is $\fS$-peaking thus reduces to distinguishing weak-$*$ exposed points among the extreme points in the set of self-adjoint functionals on $\fS$ with norm at most $1$.

The following additional remarks appear relevant in view of  Theorem \ref{T:peakingexposed}. First, a classical result of Klee \cite{klee1958} (see also \cite{CD2016predual} for a detailed proof) guarantees that weak-$*$ exposed points are plentiful. Indeed, the set of weak-$*$ exposed points of a weak-$*$ closed convex subset is weak-$*$ dense in the set extreme points. Second, we mention that in the classical case of a uniform algebra, it is known (\cite[Theorem II.11.3]{gamelin1969}) that the Choquet boundary coincide with the set of peak points. The multiplicative structure of the algebra is crucial here, as the equality of these two sets fails for general function systems \cite[page 43]{phelps2001}. Nevertheless, the peak points are still abundant; they are in fact dense in the Choquet boundary \cite[Theorem 8.4]{phelps2001}. 

We exhibit some non-commutative examples related to these ideas.

\begin{example}\label{E:peakingBH}
Let $\H$ be a Hilbert space and let $\fA\subset B(\H)$ be a unital $\rC^*$-algebra which contains the ideal $\fK$ of compact operators on $\H$. Assume that $\fA/\fK$ is not one-dimensional and put $\fS=\fK+\bC I$. Repeating the argument of \cite[Example 4.5]{clouatre2017unp}, we see that the only pure states on $\fA$ that have the unique extension property with respect to $\fS$ are the vector states on $\fA$. We proceed to verify that these vector states are always $\fS$-peaking, regardless of whether $\fA/\fK$ is one-dimensional or not.

Let $\xi\in \H$ be a unit vector and let $\omega$ be the pure vector state on $\fA$ defined as
\[
\omega(a)=\langle a\xi,\xi \rangle, \quad a\in \fA.
\]
Let $s\in \fK$ be the rank-one projection onto $\bC \xi$. Clearly, we have that $\omega(s)=1$. Let now $\phi$ be a state on $\fA$. Standard facts about the representation theory of $\rC^*$-algebras (see the discussion preceding \cite[Theorem I.3.4]{arveson1976inv}) then imply that the GNS representation $\sigma_\phi$ is unitarily equivalent to 
\[
\id^{(\gamma)}\oplus \pi_Q
\]
where $\gamma$ is some cardinal number and $\pi_Q$ is a unital $*$-representation of $\fA$ which annihilates $\fK$.  This shows that there is a positive trace class operator $T\in B(\H)$ and a positive linear functional $\theta$ on $\fA$ satisfying $\fK\subset\ker \theta$ with the property that
\[
\phi(a)=\tr (aT)+\theta(a), \quad a\in \fA.
\]
In particular, we see that $\phi(s)=\tr (sT)$, whence $|\phi(s)|<1$ unless $\theta=0$ and $T=s$, which is equivalent to $\phi=\omega$. Hence $\omega$ is $\fS$-peaking. 

\end{example}

\begin{example}\label{E:peakingmatrixunits}
Let $\bM_2$ denote the $2\times 2$ complex matrices and consider the standard matrix unit $E_{12}\in \bM_2$. Let $\fS\subset \bM_2$ be the operator system generated by $I$ and $E_{12}$.
Note that $\bM_2=\rC^*(\fS)$. It is well-known that up to unitary equivalence, the only irreducible $*$-representation of $\bM_2$ is the identity representation, so that the pure states on $\bM_2$ are the vector states. For each unit vector $\xi\in \bC^2$, we let $\omega_\xi$ be the state on $\bM_2$ defined as
\[
\omega_\xi(a)=\langle a\xi,\xi\rangle, \quad a\in \bM_2.
\]
Observe that $\omega_\xi=\omega_{\xi'}$ if and only if $\xi=\alpha\xi'$ for some $\alpha\in \bC$ with $|\alpha|=1$.

Let $\{e_1,e_2\}$ be the canonical orthonormal basis of $\bC^2$ and assume that
\[
\xi=\xi_1 e_1+\xi_2 e_2
\]
is a unit vector, where $\xi_1,\xi_2\in \bC$. We wish to determine when $\omega_\xi$ has the unique extension property with respect to $\fS$ and when it is $\fS$-peaking. We distinguish two cases.

First, suppose that $\xi_1$ is not of the form $\alpha\ol{\xi_2}$ for some $\alpha\in \bC$ with $|\alpha|=1$. Consider the unit vector
\[
\xi'=\ol{\xi_2}e_1+\ol{\xi_1}e_2.
\]
Then, $\omega_{\xi}\neq \omega_{\xi'}$. However, we note that
\[
\langle E_{12}\xi,\xi\rangle=\xi_2 \ol{\xi_1}=\langle E_{12}\xi',\xi'\rangle
\]
so that $\omega_\xi(E_{12})=\omega_{\xi'}(E_{12})$. This implies that $\omega_{\xi}|_{\fS}=\omega_{\xi'}|_{\fS}$, and we infer that $\omega_\xi$ does not have the unique extension property with respect to $\fS$, and in particular it is not $\fS$-peaking by Lemma \ref{L:peakuep}.

Second, suppose that $\xi_1=\alpha \ol{\xi_2}$ for some $\alpha\in \bC$ with $|\alpha|=1$. Then $\xi_2=\alpha\ol{\xi_1}$ and 
\[
|\xi_1|=|\xi_2|=1/\sqrt{2}.
\]
We infer
\[
\xi_1=\beta/\sqrt{2}, \quad \xi_2=\alpha\ol{\beta}/\sqrt{2}
\]
for some $\beta\in \bC$ with $|\beta|=1$.
We now claim that $\omega_\xi$ is $\fS$-peaking. Indeed, consider
\[
s=\frac{1}{2}\left(I+\ol{\alpha}\beta^2E_{12}+\alpha\ol{\beta}^2E_{21} \right)\in \fS.
\]
An easy calculation shows that $s$ is the rank-one projection onto $\bC\xi$ and thus that
\[
\omega_\xi(a)=\tr(as), \quad a\in \bM_2.
\]
It is clear that $\omega_\xi(s)=1$. Let $\phi$ be a state on $\bM_2$. Then, there is a positive element $T\in \bM_2$ such that $\tr T=1$ and
\[
\phi(a)=\tr(aT), \quad a\in \bM_2.
\]
If $|\phi(sT)|=1$, then a moment's thought reveals that we must have $s=T$ which implies that $\phi=\omega_\xi$. Thus, $\omega_\xi$ is $\fS$-peaking and in particular it has the unique extension property with respect to $\fS$ by Lemma \ref{L:peakuep}.

We have shown that every pure state on $\bM_2$ which has the unique extension property with respect to $\fS$ must be $\fS$-peaking.

\end{example}

Inspired by \cite{hay2007}, we now make the following definition. Let $\fA$ be a unital $\rC^*$-algebra and let $\fS\subset \fA$ be an operator system.  A projection $p\in \fA^{**}$ is said to be \emph{$\fS$-peaking} if there is a self-adjoint element $s\in \fS$ with $\|s\|=1$ such that $ sp=p$ and $| \phi(s)|<1$ whenever $\phi$ is a state on $\fA$ with support projection $\fs_\phi\in \fA^{**}$ orthogonal to $p$. In this case, we also have $ps=p$ and we say that $p$ \emph{peaks on} $s$.

We warn the reader that our notion of $\fS$-peaking projection is slightly different from the corresponding one given in \cite{hay2007}.  We feel our version is more natural in the context of operator systems. This deviation from the convention of \cite{hay2007} is perhaps partly justified by the fact, which we will show next, that with our definition in its current form we have the very natural equivalence between a pure state $\omega$ and its support projection $\fs_\omega$ being $\fS$-peaking. At first glance, this is not completely obvious, as an $\fS$-peaking projection only offers control over a restricted subset of states. Indeed, assuming that the projection $\fs_\omega$ is $\fS$-peaking on $s$, we know that $|\phi(s)|<1$ whenever $\phi$ is a state on $\fA$ with $\fs_\phi \fs_\omega=0$. However, to guarantee that the state $\omega$ is $\fS$-peaking, we need to know that $|\phi(s)|<1$ for every state $\phi$ on $\fA$ which is merely distinct from $\omega$. In light of Proposition \ref{P:purestateortho}, it may thus appear as though the equivalence we seek could only hold in the case where $\fA$ is commutative. Perhaps surprisingly, this is not the case. To show this fact we need the following, which is an adaptation of  \cite[Lemma 3.6]{hay2007} to our alternative definition of peaking projection. The proof is almost identical.

\begin{lemma}\label{L:weak*limit}
Let $\fA$ be a unital $\rC^*$-algebra and let $\fS\subset \fA$ be an operator system. Let $p\in \fA^{**}$ be an $\fS$-peaking projection which peaks on the self-adjoint element $s\in \fS$. Then, the sequence $(s^n)_n$ converges to $p$ in the weak-$*$ topology of $\fA^{**}$.
\end{lemma}
\begin{proof}
We may assume that $\fA^{**}\subset B(\H)$ for some Hilbert space $\H$, and that the weak-$*$ topology of $\fA^{**}$ coincides with that of $B(\H)$. Since $\|s\|=1$, it suffices to show that the sequence $(s^n)_n$ converges to $p$ in the weak operator topology of $B(\H)$. To see this, we use the fact that $p=s p=p s$ to infer
\[
s=p+(I-p)s (I-p).
\]
Put $t=(I-p)s(I-p)$ and observe that $\|t\|\leq 1$. We claim that $t$ is completely non-unitary, in the sense that it has no reducing subspace on which it restricts to be unitary. Assume on the contrary that there is a closed subspace $\M\subset (I-p)\H$ which is reducing for $t$ and such that the operator $u=t|_\M$ is unitary. Since $s$ is assumed to be self-adjoint, so is $u$ and we see that $u^2=I$. Thus, there is a vector state $\phi$ on $B(\M)$ such that $|\phi(u)|=1$. Let $\Psi$ be the normal state on $\fA^{**}$ defined as
\[
\Psi(x)=\phi(P_{\M}x|_{\M}),\quad x\in \fA^{**}.
\]
If we put $\psi=\Psi|_{\fA}$ then it is easily verified that $\psi$ is a state on $\fA$ such that $\widehat\psi=\Psi$. 
We see that 
\[
|\psi(s)|=|\phi(u)|=1.
\]
On the other hand, since $\M$ is orthogonal to the range of $p$, we see that $ \widehat\psi(p)=0$ whence $p\leq I-\fs_\psi$ and $\fs_\psi p=0$.  This contradicts the fact that $p$ peaks on $s$. Hence, we conclude that indeed $t$ is completely non-unitary. In particular, it is an absolutely continuous contraction \cite[Theorem II.6.4]{nagy2010}, so that
\[
\lim_{n\to \infty}\langle t^n \xi,\eta\rangle=0, \quad \xi,\eta\in \H.
\]
(Alternatively, since $t$ is self-adjoint one can apply the sharper  \cite[Theorem II.6.7]{nagy2010} in this context). Since $s^n=p+t^n$ for each $n\in \bN$, we conclude that the sequence $(s^n)_n$ converges to $p$ in the weak operator topology of $B(\H)$. 
\end{proof}

We can now establish a very natural link between $\fS$-peaking states and $\fS$-peaking projections.

\begin{theorem}\label{T:peakHay}
Let $\fA$ be a unital $\rC^*$-algebra and let $\fS\subset \fA$ be an operator system. Let $s\in \fS$ be a self-adjoint element with $\|s\|=1$. Then, a pure state $\omega$ on $\fA$ is $\fS$-peaking on $s$ if and only if its support projection $\fs_\omega\in \fA^{**}$ is $\fS$-peaking on $s$.
\end{theorem}
\begin{proof}
Fix a pure state $\omega$ on $\fA$. 

Assume first that $\omega$ is $\fS$-peaking on $s$. To see that $\fs_\omega$ is $\fS$-peaking on $s$ as well, it suffices to verify that $s \fs_\omega=\fs_\omega$. Now, by assumption we see that $\omega(s)=1$ and $\|s\|=1$, so that $s$ belongs to the multiplicative domain of $\widehat\omega$ and thus
\[
\widehat \omega(xs)=\widehat \omega(x), \quad x\in \fA^{**}.
\]
This shows that $I-s\in \fN_\omega$, whence $(I-s)\fs_\omega=0$ or $s\fs_\omega=\fs_\omega$.  We conclude that $\fs_\omega$ is $\fS$-peaking.

Conversely, assume that $\fs_\omega$ is $\fS$-peaking on $s$. We obtain
\[
\omega(s)=\widehat \omega(s\fs_\omega)=\widehat\omega(\fs_\omega)=1.
\]
Next, by Lemma \ref{L:weak*limit} we see that the sequence $(s ^n)_n$ converges to $\fs_\omega$ in the weak-$*$ topology of $\fA^{**}$. Let $\phi$ be a state on $\fA$ such that $|\phi(s)|=1$.  Since $\|s\|= 1$, this forces  $s$ to lie in the multiplicative domain of $\phi$. In particular, we have $|\phi(s^n)|=|\phi(s)|^n=1$ and thus
\[
|\widehat \phi (\fs_\omega)|=\lim_{n\to \infty} |\phi(s^n)|=1
\]
which implies $\phi=\omega$ by Lemma \ref{L:distinctpure}. We conclude that $\omega$ is $\fS$-peaking on $s$.
\end{proof}

It is worthy of note that in the case where $\fS$ is a unital \emph{subalgebra} of $\fA$, some faithful analogues of classical characterizations were obtained for Hay's peaking projections (see \cite{blecher2013} and references therein).

We close this section by discussing another notion of non-commutative peaking behaviour, guided by \cite{arveson2011}. Let $\fA$ be a unital $\rC^*$-algebra and let $\fS\subset \fA$ be an operator system. Let $\pi$ be a unital $*$-representation of $\fA$. Given another unital $*$-representation $\rho$ of $\fA$, we write $\pi\prec \rho$ to signify that 
\[
\|\pi(a)\|\leq \|\rho(a)\|, \quad a\in \fA.
\]
Equivalently, we see that $\pi\prec \rho$ if and only if there is a unital $*$-representation $\theta:\rho(\fA)\to \pi(\fA)$ such that $\theta\circ \rho=\pi$. 
We say that the $*$-representation $\pi$ of $\fA$ is  $\fS$-\emph{peaking} if there is an element $s\in \fS$ with the property that 
\[
\|\pi(s)\|>\|\rho(s)\|
\]
for every unital $*$-representation $\rho$ of $\fA$ such that $\pi\nprec \rho$. We then say that $\pi$ \emph{peaks on} $s$. The reader will note that this definition apparently differs from that of \cite{arveson2011}. However, in the setting where it is used therein, the definitions coincide as the irreducible $*$-representations of $\rC^*$-algebras of compact operators are particularly transparent.

The GNS construction allows one to construct $*$-representations starting from states. We show how the condition of being $\fS$-peaking is preserved by this procedure.

\begin{proposition}\label{P:peakrep}
Let $\fA$ be a unital $\rC^*$-algebra and let $\fS\subset \fA$ be an operator system. Let $\psi$ be a state on $\fA$ which is $\fS$-peaking. Then, its GNS representation $\sigma_\psi$ is $\fS$-peaking as well.
\end{proposition}
\begin{proof}
Since $\psi$ is $\fS$-peaking, there is a self-adjoint element $s\in \fS$ such that $\|s\|=1$ and with the property that 
\[
1=\psi(s)>|\phi(s)|
\]
for every state $\phi$ on $\fA$ such that $\phi\neq \psi$. Now, let $\rho$ be a unital $*$-representation such that $\sigma_\psi\not\prec \rho$. Seeing as $s$ is self-adjoint, we may choose a state $\chi$ on $\rho(\fA)$ with the property that 
\[
|\chi(\rho(s))|=\|\rho(s)\|.
\]
Note now that if $\chi\circ \rho=\psi$ then $\sigma_\chi\circ \rho$ is unitarily equivalent to $\sigma_\psi$, and thus $\sigma_\psi\prec \rho$ which is absurd. Hence, $\chi\circ \rho\neq \psi$. We infer that
\[
\|\rho(s)\|=|\chi(\rho(s))|<\psi(s)\leq \|\sigma_\psi(s)\|
\]
and therefore $\sigma_\psi$ is $\fS$-peaking.
\end{proof}

One consequence of the previous proposition is that the condition that every pure state on $\fA$ be $\fS$-peaking implies that every irreducible $*$-representation of $\fA$ is $\fS$-peaking. The converse does not hold, as we now illustrate.

\begin{example}\label{E:peakreppeakstate}
We return to the setting of Example \ref{E:peakingmatrixunits}. Let $\bM_2$ denote the $2\times 2$ complex matrices and consider the standard matrix unit $E_{12}\in \bM_2$. Let $\fS\subset \bM_2$ be the operator system generated by $I$ and $E_{12}$. We have $\bM_2=\rC^*(\fS)$. Up to unitary equivalence, the only irreducible $*$-representation of $\bM_2$ is the identity representation. Tautologically, we see that every irreducible $*$-representation of $\bM_2$ is $\fS$-peaking. However, we saw in Example \ref{E:peakingmatrixunits} that there are plenty of pure states on $\bM_2$ which are not $\fS$-peaking.

\end{example}

\section{Characteristic sequences}\label{S:aspeaking}

In the previous section, we introduced and studied the notion of $\fS$-peaking states, a certain non-commutative analogue of classical peak points for uniform algebras. For the remainder of the paper, we aim to exploit these ideas to obtain a local version of the hyperrigidity conjecture. To do so, we will need to consider states that have a global type of peaking behaviour, such as in the following example.

\begin{example}\label{E:gapeakingBH}
Let $\H$ be a Hilbert space and let $\fA\subset B(\H)$ be a unital $\rC^*$-algebra which contains the ideal $\fK$ of compact operators on $\H$. Put $\fS=\fK+\bC I$. Let $\xi\in \H$ be a unit vector and let $\omega$ be the pure vector state on $\fA$ defined as
\[
\omega(a)=\langle a\xi,\xi \rangle, \quad a\in \fA.
\]
We saw in Example \ref{E:peakingBH} that $\omega$ is $\fS$-peaking. Furthermore, if we let $s\in \fS$ be the rank-one projection onto $\bC \xi$, then a simple calculation reveals that 
\[
s^* a s=\langle a\xi,\xi\rangle s=\omega(a) s
\]
whence 
\[
\|s^* a s\|=|\omega(a)|
\]
for each $a\in \fA$. This last equality says that the $\fS$-peaking state $\omega$ also possesses a certain global peaking property on $\fA$.

\end{example}

It will be sufficient for our purposes to have such a global peaking condition be satisfied in an asymptotic sense. Before proceeding with the definition, we take a closer look at the classical setting so as to provide motivation for what is to come.

\begin{example}\label{E:globpeakunif}
Let $X$ be a compact Hausdorff space and $\A\subset C(X)$ be a unital subalgebra. Let $x_0\in X$ be a point such that there is a function $\phi_0\in\A$ with the property that 
\[
|\phi_0(y)|<\phi_0(x_0)=1
\]
for each $y\in X, y\neq x_0$. 
Then, the peak point $x_0$ has a certain asymptotic global peaking property. More precisely, we have that
\[
\lim_{n\to \infty} \|\phi_0^{n}f\|= |f(x_0)|, \quad f\in C(X).
\]
To see this, fix a non-zero function $f\in C(X)$ and a number $\eps>0$. By continuity, there is an open subset $U\subset X$  containing $x_0$ with the property that 
\[
|f(y)|< |f(x_0)|+\eps
\]
if $y\in U$. Note also that $X\setminus U$ is compact, and thus there is a positive integer $N$ large enough so that
\[
\sup\{ |\phi_0(y)|^{n}: y\in X\setminus U\} \leq \eps \|f\|^{-1}, \quad n\geq N.
\]
Consequently,
\[
\limsup_{n\to \infty}\|\phi_0^n f \|\leq |f(x_0)|+\eps.
\]
Since $\eps>0$ is arbitrary, we infer that
\[
\limsup_{n\to \infty}\|\phi_0^n f \|\leq |f(x_0)|.
\]
On the other hand, for each positive integer $n$ we have
\[
\|\phi_0^n f \|\geq |f(x_0)|
\]
and thus 
\[
\lim_{n\to \infty} \|\phi_0^{n}f\|= |f(x_0)|.
\]

\end{example}

Roughly speaking, the previous example shows that a peak point $x_0\in X$ for a uniform algebra $\A$ gives rise to a sequence lying in $\A$ which  behaves asymptotically like the characteristic function of the singleton $\{x_0\}$. With the help of this device, the uniform algebra can be ``localized" at the peak point. This is the phenomenon we want to reproduce in the general non-commutative setting, and accordingly we make the following definition, motivated by Examples \ref{E:gapeakingBH} and \ref{E:globpeakunif}.

Let $\fA$ be a unital $\rC^*$-algebra and let $\psi$ be a state on $\fA$. A sequence $(\Delta_n)_n$ in $\fA$ is said to be a \emph{characteristic sequence} for $\psi$  if the following conditions are satisfied:
\begin{enumerate}
\item[\rm{(a)}] $\|\Delta_n\|=1$ for every $n\in \bN$,
\item[\rm{(b)}] $\lim_{n\to \infty }\psi(\Delta_n)=1,$ and

\item[\rm{(c)}] $\limsup_{n\to\infty}\|\Delta_n^* a \Delta_n\|\leq |\psi(a)|$ for every $a\in \fA$.
\end{enumerate}

To uncover some of the properties of characteristic sequences, we require the following fact.

\begin{lemma}\label{L:charseqlim}
Let $\fA$ be a unital $\rC^*$-algebra and let $\phi, \psi$ be states on $\fA$. Assume that $(\Delta_n)_n$ is a characteristic sequence for $\psi$ and that 
\[
\lim_{n\to\infty}|\phi(\Delta_n)|=1.
\]
Then, we have that
\[
\phi(a)=\lim_{n\to \infty}\phi(\Delta_n^* a\Delta_n), \quad a\in \fA
\]
and $\phi=\psi$.
\end{lemma}
\begin{proof}
By the Schwarz inequality we have
\[
|\phi(\Delta_n)|^2=\phi(\Delta_n)^*\phi(\Delta_n)\leq \phi(\Delta_n^*\Delta_n)\leq 1
\]
and thus
\[
\lim_{n\to \infty}(\phi(\Delta_n^*\Delta_n) -\phi(\Delta_n)^*\phi(\Delta_n)) =0.
\]
Let $a\in \fA$. For each $n\in \bN$ consider the element
\[
A_n=
\begin{bmatrix}
\Delta_n & a^*\\
0 & 0
\end{bmatrix}\in \bM_2(\fA).
\]
By the Schwarz inequality, we find that
\[
\phi^{(2)}(A_n)^* \phi^{(2)}(A_n)\leq \phi^{(2)}(A_n^*A_n).
\]
Equivalently
\[
\begin{bmatrix}
\phi(\Delta_n)^*\phi(\Delta_n) &\phi(\Delta_n)^*\phi(a)^*\\
\phi(a)\phi(\Delta_n) & \phi(a)\phi(a)^*
\end{bmatrix}
\leq \begin{bmatrix}
\phi(\Delta_n^*\Delta_n) & \phi(\Delta_n^* a^*)\\
\phi(a\Delta_n) & \phi(aa^*)
\end{bmatrix},
\]
or
\[
\begin{bmatrix}
\phi(\Delta_n^*\Delta_n) -\phi(\Delta_n)^*\phi(\Delta_n) & \phi(\Delta_n^* a^*)-\phi(\Delta_n)^*\phi(a)^*\\
\phi(a\Delta_n)-\phi(a)\phi(\Delta_n)  & \phi(aa^*)- \phi(a)\phi(a)^*
\end{bmatrix}\geq 0.
\]
In particular, we find that
\begin{align*}
|\phi(a\Delta_n)-\phi(a)\phi(\Delta_n) |^2&\leq ( \phi(aa^*)- \phi(a)\phi(a)^*)(\phi(\Delta_n^*\Delta_n) -\phi(\Delta_n)^*\phi(\Delta_n))\\
&\leq \|a\|^2(\phi(\Delta_n^*\Delta_n) -\phi(\Delta_n)^*\phi(\Delta_n)).
\end{align*}
Repeating the argument with $a^*$ in place of $a$ yields
\begin{align*}
|\phi(\Delta_n^*a )-\phi(a)\phi(\Delta^*_n) |^2&=|\phi(a^*\Delta_n)-\phi(a^*)\phi(\Delta_n) |^2\\
&\leq \|a\|^2(\phi(\Delta_n^*\Delta_n) -\phi(\Delta_n)^*\phi(\Delta_n)).
\end{align*}
Therefore
\begin{align*}
&|\phi(\Delta_n^*a\Delta_n)-\phi(\Delta_n^*)\phi(a)\phi(\Delta_n) |\\
&\leq |\phi(\Delta_n^*a\Delta_n)-\phi(\Delta_n^*a)\phi(\Delta_n)|+|\phi(\Delta_n^*a)\phi(\Delta_n)-\phi(\Delta_n^*)\phi(a)\phi(\Delta_n)|\\
&\leq (\|\Delta_n^* a\| +|\phi(\Delta_n)|\|a\|)(\phi(\Delta_n^*\Delta_n) -\phi(\Delta_n)^*\phi(\Delta_n))^{1/2}\\
&\leq 2\|a\|(\phi(\Delta_n^*\Delta_n) -\phi(\Delta_n)^*\phi(\Delta_n))^{1/2}
\end{align*}
which then forces
\[
\lim_{n\to \infty}(\phi(\Delta_n^* a\Delta_n)-\phi(\Delta_n^*)\phi(a)\phi(\Delta_n) )=0.
\]
Since 
\[
\lim_{n\to \infty}|\phi(\Delta_n)|=1
\]
we find
\[
\phi(a)=\lim_{n\to \infty}\phi(\Delta_n^* a\Delta_n)
\]
and thus
\[
|\phi(a)|\leq |\psi(a)|
\]
since $(\Delta_n)_n$ is a characteristic sequence for $\psi$. Since this holds for every $a\in \fA$, we infer that $\ker \psi\subset \ker \phi$, so that $\phi$ is a scalar multiple of $\psi$. Since they are both states, this forces $\phi=\psi$. 
\end{proof}

We can now establish some useful properties of characteristic sequences.

\begin{theorem}\label{T:gapeakSchwarz}
Let $\fA$ be a unital $\rC^*$-algebra and let $\psi$ a state on $\fA$. Let $(\Delta_n)_n$ be a characteristic sequence for $\psi$. 
The following statements hold.
\begin{enumerate}

\item[\rm{(1)}]  We have that 
\[
\lim_{n\to\infty}\|\Delta_n a \Delta_n^*\|=|\psi(a)|
\]
for every $a\in \fA$.

\item[\rm{(2)}]  We have that 
\[
\limsup_{n\to\infty}|\phi(\Delta_n)|<1
\]
for every state $\phi$ on $\fA$ such that $\phi\neq \psi$.

\item[\rm{(3)}] The state $\psi$ is pure. 

\item[\rm{(4)}] Let $\fS\subset \fA$ be an operator system and assume that $\Delta_n\in \fS$ for each $n\in \bN$. Then, $\psi$ has the unique extension property with respect to $\fS$.
\end{enumerate}
\end{theorem}
\begin{proof}
To show (1), fix $a\in \fA$. By definition of a characteristic sequence, we have that 
\[
\lim_{n\to \infty }\psi(\Delta_n)=1
\]
and
\[
\limsup_{n\to\infty}\|\Delta_n^* a \Delta_n\|\leq |\psi(a)|.
\]
On the other hand, invoking Lemma \ref{L:charseqlim}, we obtain
\[
|\psi(a)|=\lim_{n\to \infty}|\psi(\Delta_n^* a\Delta_n)|\leq \liminf_{n\to\infty}\|\Delta_n^* a \Delta_n\|
\]
whence
\[
\lim_{n\to\infty}\|\Delta_n a \Delta_n^*\|=|\psi(a)|.
\]

Next, let $\phi$ be a state on $\fA$ with $\phi\neq \psi$ and let $(\Delta_{n_m})_m$ be a subsequence. It is readily seen that $(\Delta_{n_m})_m$ is still a characteristic sequence for $\psi$. Thus, by virtue of Lemma \ref{L:charseqlim}, we see that
\[
\lim_{m\to\infty}|\phi(\Delta_{n_m})|<1.
\]
Consequently, we see that
\[
\limsup_{n\to\infty}|\phi(\Delta_n)|<1
\]
and (2) holds.

Assume now that 
\[
\psi=\frac{1}{2}(\phi_1+\phi_2)
\]
for some states $\phi_1,\phi_2$ on $\fA$. Since
\[
\lim_{n\to\infty} \psi(\Delta_n)=1
\]
a routine verification yields that
\[
\lim_{n\to\infty} |\phi_1(\Delta_n)|=1=\lim_{n\to\infty} |\phi_2(\Delta_n)|.
\]
Therefore $\phi_1=\phi_2=\psi$ by (2), and thus $\psi$ is pure. We have established (3). 

For (4), we assume that $\Delta_n\in \fS$ for each $n\in \bN$ and we let $\phi$ be a state on $\fA$ which agrees with $\psi$ on $\fS$. Then
\[
\lim_{n\to \infty}|\phi(\Delta_n)|=\lim_{n\to \infty }|\psi(\Delta_n)|=1
\]
so that $\phi=\psi$ by part (2). We conclude that $\psi$ has the unique extension property with respect to $\fS$.
\end{proof}

The previous result shows that if a state $\psi$ on $\fA$ admits a characteristic sequence in $\fS$, then $\psi$ shares several properties with $\fS$-peaking states. Indeed, (3) and (4) correspond to Lemma \ref{L:peakuep}, while (2) says that $\psi$ can be considered to be  ``asymptotically $\fS$-peaking" on the sequence $(\Delta_n)_n$. Note however that typically a characteristic sequence does not consist of self-adjoint elements (see Example \ref{E:Toeplitz} below). We mention an outstanding problem that we have not been able to settle.

\begin{question}
Let $\fA$ be a unital $\rC^*$-algebra and let $\fS\subset \fA$ be an operator system. Let $\psi$ be a pure state on $\fA$ that has the unique extension property with respect to $\fS$. Does there always exist a characteristic sequence in $\fS$ for $\psi$?
\end{question}

A positive answer (even in the commutative case) would constitute a nice analogue of the classical fact that for uniform algebras, the Choquet boundary coincides with the set of peak points  \cite[Theorem II.11.3]{gamelin1969}. As mentioned before, this is known to fail for general function systems \cite[page 43]{phelps2001}. It is unclear to us at the moment whether the asymptotic peaking behaviour of characteristic sequences is flexible enough to remedy this defect. 

We close this section by answering the previous question in the affirmative in some special cases. For that purpose, recall that an operator $T\in B(\H)$ is said to be a \emph{pure contraction} if $\|T\|\leq 1$ and
\[
\lim_{n\to \infty}\|T^{*n}\xi\|=0, \quad \xi\in \H.
\]

\begin{theorem}\label{T:globpeakingessab}
Let $\H$ be a Hilbert space. Let $\fA\subset B(\H)$ be a unital $\rC^*$-algebra which contains the ideal of compact operators $\fK$ and such that the quotient $\fA/\fK$ is commutative. Let $\A\subset \fA$ be a unital closed subalgebra such that for every character $\omega$ on $\fA$, there is a pure contraction $t_\omega\in \A$ with the property that
\[
|\chi(t_\omega)|<1=\omega(t_\omega)
\]
whenever $\chi$ is a character on $\fA$ such that $\chi\neq\omega$.
Then, every pure state on $\fA$ admits a characteristic sequence in $\fS+\fK$, where $\fS$ is the operator system generated by $\A$.
\end{theorem}
\begin{proof}
Let $\omega$ be a pure state on $\fA$. Based on the discussion from Example \ref{E:peakingBH}, we see that $\omega$ is either a vector state or it annihilates $\fK$.  It also follows from Example \ref{E:gapeakingBH} that every vector state on $\fA$ admits a (constant) characteristic sequence in $\fK$. 
Next, assume $\omega$ annihilates $\fK$. There is a pure state $\widehat\omega$ on $\fA/\fK$ with the property that if $q:\fA\to \fA/\fK$ denotes the quotient map, then $\omega=\widehat\omega\circ q$. Since $\fA/\fK$ is a commutative $\rC^*$-algebra, $\widehat \omega$ must be a character and thus so is $\omega$. By assumption, there is a pure contraction $t_\omega\in \A$ with the property that
\[
|\chi(t_\omega)|<1=\omega(t_\omega)
\]
for every character $\chi$ on $\fA$ such that $\chi\neq\omega$.  In particular, this means that
\begin{equation}\label{E:charineq}
|\theta(q(t_\omega))|<1=\widehat \omega(q(t_\omega))
\end{equation}
for every character $\theta$ on $\fA/\fK$ such that $\theta \neq \widehat\omega$.
We now verify that 
\[
\Delta_n=t_\omega^{n}, \quad n\in \bN
\]
is a characteristic sequence for $\omega$ in $\fS$. Clearly, we have that $\omega(\Delta_n)=1$ for every $n\in \bN$. Moreover, Example \ref{E:globpeakunif} along with Equation (\ref{E:charineq}) implies that
\[
\lim_{n\to \infty}\|q(\Delta^*_n a \Delta_n)\|=  |\omega(a)|, \quad a\in \fA.
\]
Fix $a\in\fA$ and let $\eps>0$. Choose $N\in \bN$ with the property that
\[
\|q(\Delta^*_N a \Delta_N)\|< |\omega(a)|+\eps.
\]
We can find a compact operator $K$ such that
\[
\|\Delta^*_N a \Delta_N+K\|< |\omega(a)|+\eps.
\]
Since $\|\Delta_m\|=1$ and $\Delta_N \Delta_m=\Delta_{N+m}$ for every $m\in \bN$, we conclude that
\[
\|\Delta^*_{N+m} a \Delta_{N+m}+\Delta^*_m K\Delta_m\|< |\omega(a)|+\eps, \quad m\in \bN.
\]
Now, the sequence $(\Delta^*_m)_m$ converges to $0$ in the strong operator topology as $t_\omega$ is assumed to be a pure contraction. Since $K$ is compact, the sequence $(\Delta^*_m K\Delta_m)_m$ must converge to $0$ in norm and we infer
\[
\limsup_{m\to \infty}\|\Delta^*_{N+m} a \Delta_{N+m}\|\leq |\omega(a)|+\eps.
\]
Since $\eps>0$ is arbitrary, we conclude that $(\Delta_n)_n$ is a characteristic sequence for $\omega$.
\end{proof}

We now exhibit a natural example satisfying the conditions of the previous theorem.

\begin{example}\label{E:Toeplitz}
Fix a positive integer $d\geq 1$. Let $\bB_d\subset \bC^d$ denote the open unit ball and let $\bS_d$ denotes its boundary, the unit sphere. The \emph{Drury-Arveson space} $H^2_d$ is the reproducing kernel Hilbert space on $\bB_d$ with reproducing kernel given by the formula
\[
k(z,w)=\frac{1}{1-\langle z,w\rangle_{\bC^d}}, \quad z,w\in \bB_d.
\] 
This is a Hilbert space of holomorphic functions on $\bB_d$. If for each $\lambda\in \bB_d$ we put
\[
k_\lambda(z)=\frac{1}{1-\langle z,\lambda\rangle_{\bC^d}}, \quad z\in \bB_d
\]
then
\[
f(\lambda)=\langle f,k_\lambda \rangle_{H^2_d}, \quad f\in H^2_d.
\]
A function $\phi:\bB_d\to \bC$ is a \emph{multiplier} for $H^2_d$ if $\phi f\in H^2_d$ for every $f\in H^2_d$. Examples of such functions include the holomorphic polynomials in $d$ variables.  Every multiplier $\phi$ gives rise to a multiplication operator $M_\phi\in B(H^2_d)$, and the identification $\phi\mapsto M_\phi$ allows us to view the algebra of multipliers as an operator algebra on $H^2_d$. It is easily verified that if $\phi$ is a multiplier, then
\[
M_\phi^* k_\lambda=\ol{\phi(\lambda)}k_\lambda, \quad \lambda\in \bB_d.
\]
We refer the reader to the book \cite{agler2002} for an excellent reference on these topics.

Let $\A_d\subset B(H^2_d)$ denote the norm closure of the polynomial multipliers and define the \emph{Toeplitz algebra} as $\fT_d=\rC^*(\A_d)$.  Recall that $\fT_d$ contains the ideal of compact operators $\fK$ on $H^2_d$ and that $\fT_d/\fK$ is $*$-isomorphic to $C(\bS_d)$ \cite[Theorem 5.7]{arveson1998}. In particular, any character on $\fT_d/\fK$ can be identified with the character of evaluation at some point in $\bS_d$.

Let $\omega$ be a character on $\fT_d$, which necessarily annihilates $\fK$. There is a character $\widehat\omega$ on $\fT_d/\fK$ with the property that if $q:\fT_d\to \fT_d/\fK$ denotes the quotient map, then $\omega=\widehat\omega\circ q$. 
Assume that $\widehat\omega$ is identified with the character of evaluation at $\zeta\in \bS_d$. Consider the function defined as
\[
\phi_\zeta(z)=\frac{1}{2}(1+\langle z,\zeta\rangle_{\bC^d}), \quad z\in \ol{\bB_d}.
\]
The row operator $(M_{z_1},M_{z_2},\ldots,M_{z_d})$ is contractive, so that $M_{\phi_\zeta}\in \A_d$ and $\|M_{\phi_\zeta}\|=1$. Moreover, $\phi_\zeta(\zeta)=1$ and $|\phi_\zeta(z)|<1$ for every $z\in \ol{\bB_d}$ with $z\neq \zeta$. In particular, we see that 
\[
\lim_{n\to \infty}M_{\phi_\zeta}^{*n}k_\lambda=0, \quad \lambda\in\bB_d.
\]
Since  the subset $\{k_\lambda:\lambda\in \bB_d\}$ spans a dense subspace of $H^2_d$, we conclude that $M_{\phi_\zeta}$ is a pure contraction. It follows from Theorem \ref{T:globpeakingessab} that every pure state on $\fT_d$ admits a characteristic sequence in $\fK+\fS_d$, where $\fS_d$ is the operator system generated by $\A_d$.

\end{example}

We note in passing that the Toeplitz algebra in the previous example is of prime importance in multivariate operator theory, dilation theory and function theory, as exposed in \cite{arveson1998} and \cite{AM2000} (see also the references therein). Furthermore, we mention that by virtue of \cite[Proposition 2.4]{kennedyshalit2015}, the argument above can easily adapted to the setting of quotient modules of the Drury-Arveson space for which the essential normality conjecture is known to hold \cite{englis2015}, \cite{DTY2016}.

\section{A local version of the hyperrigidity conjecture}\label{S:local}

Recall that Theorem \ref{T:arvlochr} exhibits, for commutative $\rC^*$-algebras, a local version of hyperrigidity using the spectral measure associated to a separable $*$-representation. Our aim in this final section is to generalize this result to arbitrary $\rC^*$-algebras. In the absence of a spectral measure, a key step in this endeavour is to give an appropriate interpretation of what local hyperrigidity should be. To do this, we capitalize on the foundations that were laid in Section \ref{S:aspeaking} and perform a localization procedure with the help of characteristic sequences. We explain at the end of the section how this is equivalent to Theorem \ref{T:arvlochr} when we specialize our result to uniform algebras.

The main technical observation is the following.

\begin{lemma}\label{L:gapineq}
Let $\fA$ be a unital  $\rC^*$-algebra and let $\fS\subset \fA$ be an operator system. Let $\pi:\fA\to B(\H)$ be a unital $*$-representation and let $\Pi:\fA\to B(\H)$ be a unital completely positive extension of $\pi|_{\fS}$. Let $\psi$ be a state on $\fA$ which admits a characteristic sequence $(\Delta_n)_n$ in $\fS$. Then, we have
\[
\limsup_{n\to \infty}\| \Pi(a)\pi(\Delta_n)\|\leq \psi(a^*a)^{1/2}
\]
for every $a\in \fA$. Furthermore, we have
\[
\lim_{n\to\infty}\|\pi(\Delta_n)^*\Pi(a)\pi(\Delta_n)\|=0
\]
for every self-adjoint element $a\in \fA$ such that $\psi(a)=0$.
\end{lemma}
\begin{proof}
Fix $a\in \fA$. If $s\in \fS$ is a self-adjoint element that satisfies $s\geq a^*a$, then since $\Pi$ and $\pi$ coincide on $\fS$ we find
\[
\pi(\Delta_n)^*\Pi(a^*a)\pi(\Delta_n)\leq \pi(\Delta_n)^*\pi(s)\pi(\Delta_n)=\pi(\Delta_n^* s\Delta_n)
\] 
whence
\[
0\leq \pi(\Delta_n)^*\Pi(a)^*\Pi(a)\pi(\Delta_n)\leq \pi(\Delta_n^* s\Delta_n)
\]
for every $n\in \bN$, by the Schwarz inequality. In particular, we find
\[
\| \Pi(a)\pi(\Delta_n)\|^2\leq \|\Delta_n^* s\Delta_n\|, \quad n\in \bN
\]
whenever $s\in \fS$ satisfies $s\geq a^*a$. Using that $(\Delta_n)_n$ is a characteristic sequence for $\psi$, we infer that
\[
\limsup_{n\to \infty}\|\Pi(a)\pi(\Delta_n)\|\leq \psi(s)^{1/2}
\]
for every $s\in \fS$ such that $s\geq a^*a$. Consequently,
\[
\limsup_{n\to \infty}\|\Pi(a)\pi(\Delta_n)\|\leq \inf \{\psi(s):s\in \fS, s\geq a^* a\}^{1/2}.
\]
 Invoke now Theorem \ref{T:gapeakSchwarz} to see that $\psi$ has the unique extension property with respect to $\fS$. It follows from a straightforward adaptation of \cite[Proposition 6.2]{arveson2011} that
\[
\inf \{\psi(s):s\in \fS, s\geq a^* a\}=\psi(a^*a)
\]
and thus
\[
\limsup_{n\to \infty}\| \Pi(a)\pi(\Delta_n)\|\leq \psi(a^*a)^{1/2}.
\]
We have established the first statement.
 
To establish the second, fix a self-adjoint element $a\in \fA$ such that $\psi(a)=0$ and a number $\eps>0$. Then, we can find self-adjoint elements $b_\eps,c_\eps\in \fS$ such that $c_\eps\leq a\leq b_\eps$ and
\[
\sup\{\psi(c):c\in \fS, c\leq a\}\leq \psi(c_\eps)+\eps,
\]
\[
\inf\{\psi(b):b\in \fS, b\geq a\}\geq \psi(b_\eps)-\eps.
\]
Next, note that since $(\Delta_n)_n$ is a characteristic sequence for $\psi$, we can find $N\in \bN$ such that
\[
\|\Delta_n^* c_\eps \Delta_n\|\leq |\psi(c_\eps)|+\eps, \quad \|\Delta_n^* b_\eps \Delta_n\|\leq |\psi(b_\eps)|+\eps
\]
if $n\geq N$. In particular, we see that
\[
-(|\psi(c_\eps)|+\eps)I\leq \pi(\Delta_n^* c_\eps \Delta_n), \quad \pi( \Delta_n^* b_\eps \Delta_n)\leq (|\psi(b_\eps)|+\eps)I
\]
if $n\geq N$.
On the other hand, using that $\Pi$ and $\pi$ coincide on $\fS$, we note that
\[
\pi(\Delta_n^*)\pi( c_\eps) \pi(\Delta_n)\leq \pi(\Delta_n)^*\Pi(a)\pi(\Delta_n)\leq \pi(\Delta_n^*)\pi( b_\eps) \pi(\Delta_n), \quad n\in \bN
\]
or
\[
\pi(\Delta_n^* c_\eps\Delta_n)\leq \pi(\Delta_n)^*\Pi(a)\pi(\Delta_n)\leq \pi(\Delta_n^* b_\eps \Delta_n), \quad n\in \bN.
\]
Hence, we find
\[
-(|\psi(c_\eps)|+\eps)I\leq \pi(\Delta_n)^*\Pi(a)\pi(\Delta_n)\leq  (|\psi(b_\eps)|+\eps)I, \quad n\geq N.
\]
Recall now that $\psi(a)=0$ so that $\psi(b_\eps)\geq 0$ and $\psi(c_\eps)\leq 0$, and therefore
\[
(\psi(c_\eps)-\eps)I\leq \pi(\Delta_n)^*\Pi(a)\pi(\Delta_n)\leq  (\psi(b_\eps)+\eps)I, \quad n\geq N,
\]
which in turn implies
\[
(\sup\{\psi(c):c\in \fS, c\leq a\}-2\eps)I \leq \pi(\Delta_n)^*\Pi(a)\pi(\Delta_n)
\]
and
\[
\pi(\Delta_n)^*\Pi(a)\pi(\Delta_n)\leq (\inf\{\psi(b):b\in \fS, b\geq a\}+2\eps)I
\]
if $n\geq N$, because of the choice of $b_\eps$ and $c_\eps$. As above, by Theorem \ref{T:gapeakSchwarz} we see that $\psi$ has the unique extension property with respect to $\fS$, and it follows from \cite[Proposition 6.2]{arveson2011} that
\[
\sup\{\psi(c):c\in \fS, c\leq a\}=0=\inf\{\psi(b):b\in \fS, b\geq a\}.
\]
Therefore,
\[
-2\eps\leq \pi(\Delta_n)^*\Pi(a)\pi(\Delta_n)\leq  2\eps I, \quad n\geq N.
\]
Since $\eps>0$ was arbitrary, we conclude that 
\[
\lim_{n\to\infty}\pi(\Delta_n)^*\Pi(a)\pi(\Delta_n)=0.
\]

\end{proof}

We now arrive at one of the main results of the paper, which yields a local form of the hyperrigidity conjecture for general operator systems. Recall that given a state $\psi$, we denote by $\sigma_\psi$ the associated GNS representation.

\begin{theorem}\label{T:gapeaklocalhr}
Let $\fA$ be a unital  $\rC^*$-algebra and let $\fS\subset \fA$ be an operator system. Let $\pi:\fA\to B(\H)$ be a unital $*$-representation and let $\Pi:\fA\to B(\H)$ be a unital completely positive extension of $\pi|_{\fS}$. Let $\psi$ be a state on $\fA$ which admits a characteristic sequence $(\Delta_n)_n$ in $\fS$. Then, we have
\[
\limsup_{n\to\infty} \|(\Pi(a)-\pi(a))\pi(\Delta_n)\|\leq 2\dist(\sigma_\psi(a),\sigma_\psi(\fS))
\]
and
\[
\lim_{n\to\infty} \|\pi(\Delta_n)^*(\Pi(a)-\pi(a))\pi(\Delta_n)\|=0
\]
for every $a\in \fA$.
\end{theorem}
\begin{proof}
Let $a\in \fA$. Since $\Pi$ and $\pi$ coincide on $\fS$, we note that for every $s\in \fS$ we have
\[
\Pi(a)-\pi(a)=\Pi(a-s)-\pi(a-s).
\]
On the other hand, by virtue of Lemma \ref{L:gapineq} we see for each $s\in \fS$ that
\[
\limsup_{n\to\infty}\|\Pi(a-s)\pi(\Delta_n)\|\leq \psi((a-s)^*(a-s))^{1/2}\leq \|\sigma_\psi(a-s)\|
\]
and
\[
\limsup_{n\to\infty}\|\pi(a-s)\pi(\Delta_n)\|\leq \psi((a-s)^*(a-s))^{1/2}\leq \|\sigma_\psi(a-s)\|,
\]
so that
\[
\limsup_{n\to\infty} \|(\Pi(a)-\pi(a))\pi(\Delta_n)\|\leq 2\|\sigma_\psi(a-s)\|.
\]
Hence
\[
\limsup_{n\to\infty} \|(\Pi(a)-\pi(a))\pi(\Delta_n)\|\leq 2 \dist(\sigma_\psi(a),\sigma_\psi(\fS)).
\]
We have established the first statement. Since $\fA$ is spanned algebraically by its self-adjoint elements, to establish the second statement it is no loss of generality to assume that $a\in \fA$ is self-adjoint. Furthermore, because
\[
\Pi(a)-\pi(a)=\Pi(a-\psi(a)I)-\pi(a-\psi(a)I)
\]
we may also assume without loss of generality that $\psi(a)=0$. By Lemma \ref{L:gapineq}, we conclude that
\[
\lim_{n\to\infty} \|\pi(\Delta_n)^*\Pi(a)\pi(\Delta_n)\|=0=\lim_{n\to\infty} \|\pi(\Delta_n)^*\pi(a)\pi(\Delta_n)\|
\]
which implies
\[
\lim_{n\to\infty} \|\pi(\Delta_n)^*(\Pi(a)-\pi(a))\pi(\Delta_n)\|=0.
\]
\end{proof}

In view of Theorem \ref{T:globpeakingessab} and Example \ref{E:Toeplitz}, one can now easily write down concrete consequences of Theorem \ref{T:gapeaklocalhr}. We leave the details to the reader, which should compare them with those in \cite[Section 6]{CH2016}.

To conclude the paper, we wish to examine Theorem \ref{T:gapeaklocalhr} carefully in the case where $\fA$ is commutative. First, we note that in this case, a pure state $\omega$ on $\fA$ must be a character, and so $\omega$ must coincide with its GNS representation $\sigma_\omega$. Therefore, 
\[
\sigma_\omega(\fA)=\sigma_\omega(\fS)=\bC
\]
and
\[
\dist(\sigma_\omega(a),\sigma_\omega(\fS))=0, \quad a\in \fA
\]
which simplifies the first conclusion of Theorem \ref{T:gapeaklocalhr} to read 
\[
\lim_{n\to\infty} \|(\Pi(a)-\pi(a))\pi(\Delta_n)\|=  0, \quad a\in \fA.
\] 
We now explain how this is equivalent to Theorem \ref{T:arvlochr} for operator systems generated by uniform algebras.

Let $(X,\rho)$ be a compact metric space, let $\fA=C(X)$ and let $\A\subset \fA$ be a closed unital subalgebra. Put $\fS=\A+\A^*$. Let $x_0\in X$ be a point in the Choquet boundary of $\A$, so that $x_0$ is a peak point for $\A$.  There is a function $\phi_0\in \A$ with the property that
\[
|\phi_0(y)|<1=\phi_0(x_0)
\]
for every $y\in X$ such that $y\neq x_0$.
Let $\omega$ be the pure state on $\fA$ of evaluation at $x_0$. If we let $\Delta_n=\phi_0^n\in \A$ for each $n\in \bN$, then we saw in Example \ref{E:globpeakunif} that $(\Delta_n)_n$ is a characteristic sequence in $\fS$ for $\omega$.

Now, let $\H$ be a separable Hilbert space and let $\pi:\fA\to B(\H)$ be a unital $*$-representation.  The reader should refer to Subsection \ref{SS:hyper} for the notation that is used below.
Fix $\delta>0$. Given $\eps>0$, by compactness of $X$ we may find $N\in \bN$ such that 
\[
|\Delta_n(y)|^2<\eps
\]
if $n\geq N$ and $y\in X$ satisfies $\rho(y,x_0)\geq \delta$. In particular this implies 
\[
\pi(\Delta_n)  \pi(\Delta_n) ^* \leq (E_\pi(x_0,\delta)+\eps I)^2, \quad n\geq N.
\]
Hence, for every $T\in B(\H)$ we have
\[
T\pi(\Delta_n)   \pi(\Delta_n) ^*T^*\leq T (E_\pi(x_0,\delta)+\eps I)^2T^*, \quad n\geq N
\]
and therefore
\[
\|T\pi(\Delta_n) \|\leq \|TE_\pi(x_0,\delta)\|+\eps \|T\|, \quad n\geq N.
\]
Since $\eps>0$ was arbitrary we obtain
\begin{equation}\label{E:specmeasure1}
\limsup_{n\to \infty}\|T\pi(\Delta_n) \|\leq \inf_{\delta>0}\|TE_\pi(x_0,\delta)\|, \quad T\in B(\H).
\end{equation}
Conversely,  fix $n\in \bN$. Given $0<\eps<1$ we may find $\delta_0$ with the property that 
\[
|\Delta_n(y)|\geq 1-\eps
\]
if $\rho(y,x_0)<\delta_0$. Hence, if $\delta<\delta_0$ then we see that
\[
(1-\eps)^2E_\pi(x_0,\delta)\leq \pi(\Delta_n)\pi(\Delta_n)^*.
\]
Hence, for every $T\in B(\H)$ we have
\[
(1-\eps)^2TE_\pi(x_0,\delta) T^*\leq T\pi(\Delta_n)   \pi(\Delta_n) ^*T^*, \quad \delta<\delta_0
\]
and therefore
\[
\|TE_\pi(x_0,\delta)\|\leq (1-\eps)^{-1} \|T\pi(\Delta_n) \|, \quad \delta<\delta_0.
\]
Since $0<\eps<1$ was arbitrary we obtain
\begin{equation}\label{E:specmeasure2}
\limsup_{\delta\to 0}\|TE_\pi(x_0,\delta)\|\leq \inf_{n\in \bN} \|T\pi(\Delta_n) \|, \quad T\in B(\H).
\end{equation}
Using Equations (\ref{E:specmeasure1}) and (\ref{E:specmeasure2}) we see that Theorems \ref{T:gapeaklocalhr} and  \ref{T:arvlochr} are equivalent in the case of uniform algebras, so that the main result of this section is a genuine generalization of Arveson's local hyperrigidity theorem.

\bibliography{/Users/raphaelclouatre/Dropbox/Research/bibliography/biblio_main}
\bibliographystyle{plain}

\end{document}